\documentclass[reqno,12pt,a4paper]{amsart}
\usepackage{pgfplots}
\usepackage[english]{babel}
\usepackage{amsmath,amsthm,amssymb,amsfonts, mathdots}
\usepackage{scrtime, cancel}
\usepackage{enumitem, color}

\usepackage{mathtools}
\usepackage{ifpdf}
\ifpdf
\usepackage[backref]{hyperref}
\else
\usepackage[hypertex]{hyperref}
\fi
\usepackage{pdfsync,verbatim}
\usepackage{color}
\mathtoolsset{showonlyrefs}

\numberwithin{equation}{section}   

\allowdisplaybreaks
\newtheorem{theorem}{Theorem}[section]
\newtheorem{lemma}[theorem]{Lemma}
\newtheorem{proposition}[theorem]{Proposition}
\newtheorem{corollary}[theorem]{Corollary}

\theoremstyle{definition}

\newcommand{\R}{\mathbb{R}}

\newcommand{\N}{\mathbb{N}}

\newcommand{\ellipses}{E}
\def\misgausskd{d \gamma_\infty}
\def\misgaussk{\gamma_\infty}
\DeclareMathOperator{\tr}{tr}             
\newcount\dotcnt\newdimen\deltay
\def\Ddot#1#2(#3,#4,#5,#6){\deltay=#6\setbox1=\hbox to0pt{\smash{\dotcnt=1
\kern#3\loop\raise\dotcnt\deltay\hbox to0pt{\hss#2}\kern#5\ifnum\dotcnt<#1
\advance\dotcnt 1\repeat}\hss}\setbox2=\vtop{\box1}\ht2=#4\box2}

\newcommand{\Red}{\color{red}}

\title[Ornstein-Uhlenbeck maximal operator ]{
On the maximal operator\\ of a general Ornstein--Uhlenbeck  semigroup }
\author{Valentina Casarino}
\address{Universit\`a degli Studi di Padova\\Stradella san Nicola 3 \\I-36100 Vicenza \\ Italy}
\email{valentina.casarino@unipd.it}
\author{Paolo Ciatti}
\address{Universit\`a degli Studi di Padova\\Via Marzolo 9 \\I-35100 Padova \\ Italy}
\email{paolo.ciatti@unipd.it}
\author{Peter Sj\"ogren}
\address{Mathematical Sciences,  University of Gothenburg and  Mathematical Sciences\\ \hspace*{9pt} Chalmers University of Technology  \\ SE - 412 96 G\"oteborg, Sweden}
\email{peters@chalmers.se}

\thanks{The first and the second author were partially supported by GNAMPA (Project 2018 
``Operatori e disuguaglianze integrali in spazi con simmetrie") 
and MIUR (PRIN 2016 ``Real and Complex Manifolds: Geometry, Topology
and Harmonic Analysis").
This research was carried out while the third author was a Visiting Scientist at the University of Padova,  Italy, and he is grateful for its hospitality.}
\date{\today, \thistime}

\begin{document}

\begin{abstract}
If
 $Q$ is  a real, symmetric and positive definite $n\times n$ matrix, and
$B$ a real $n\times n$   matrix whose eigenvalues have negative real parts,
we consider the Ornstein--Uhlenbeck
semigroup  on $\R^n$ with covariance $Q$
and
 drift matrix $B$.
Our main result says that the associated maximal operator
is of weak type $(1,1)$
with respect to the invariant measure.
The proof has a geometric gist and hinges on the
 ``forbidden zones method'' previously introduced by the third author.

\end{abstract}

\maketitle

\section{Introduction}\label{intro}
In this paper we prove a weak type $(1,1)$ theorem for the maximal operator associated to a general Ornstein--Uhlenbeck semigroup. We extend the
proof given by the third author in 1983 in a symmetric context. Our setting is  the following.

In $\R^n$ we will consider the semigroup generated by the elliptic operator
\[\mathcal L=
\frac12
\sum_{i,j=1}^n
q_{ij}
\frac{\partial^2}{\partial x_i \partial x_j}
+
\sum_{i,j=1}^n
b_{ij}
x_i\frac{\partial}{ \partial x_j},
\]
or, equivalently,
\[\mathcal L=
\frac12
\mathrm{tr}
\big( Q\nabla^2 \big)+\langle Bx, \nabla \rangle,\]
where $\nabla$ is the gradient and $\nabla^2$ the Hessian. 
Here $Q = (q_{ij})$ is  a real, symmetric and positive definite $n\times n$
 matrix, indicating the covariance of $\mathcal L$.
The  real $n\times n$   matrix $B = (b_{ij})$ is  negative in the sense that
all its eigenvalues have negative real parts, 
and it gives the drift of  $\mathcal L$.

The semigroup is formally
$
\mathcal H_t=e^{t\mathcal L}$,
${t> 0}$, 
but to  write it more explicitly
we first introduce the 
positive definite, symmetric matrices
\begin{equation}\label{defQt}
Q_t=\int_0^t e^{sB}Qe^{sB^*}ds, \qquad \text{ $0<t\leq +\infty$},
\end{equation}
and the 
normalized Gaussian measures in $\R^n$  $\gamma_t $, with $t\in (0,+\infty]$,   having density
$$
y\mapsto (2\pi)^{-\frac{n}{2}}
(\text{det} \, Q_t)^{-\frac{1}{2} }
\exp\left({-\frac12 \langle Q_t^{-1}y,y\rangle}\right)
$$
with respect to Lebesgue measure.
Then for functions $f$ 
 in the space
of bounded continuous functions in $\R^n$
one has
\begin{equation}\label{Kolmo}
\mathcal H_t
f(x)=
\int  
f(e^{tB}x-y)d\gamma_t (y)\,, \quad x\in\R^n\,,
\end{equation}
a formula due to Kolmogorov.
The measure $\gamma_\infty$ is invariant under the action of  $\mathcal H_t$;
it will be our basic measure, replacing Lebesgue measure.

We remark that  $\big(
\mathcal H_t
\big)_{t> 0}$ 
is the transition semigroup of the stochastic process 
\[\chi (x,t)=
e^{tB}+
\int_0^t
e^{(t-s)B}
\,
dW(s),\]
where $W$ is a Brownian motion in $\R^n$ with covariance $Q$.

 We are interested in  the
 maximal operator defined as
 \begin{align*}
\mathcal H_*
f(x)=
\sup_{t> 0}
\big|
\mathcal H_t
f(x)
\big|.
\end{align*}
Under the above assumptions on $Q$ and $B$,
our main result is the following.
\begin{theorem}\label{weaktype1}
The Ornstein--Uhlenbeck maximal operator
$\mathcal H_*$
 is of weak type $(1,1)$ with respect to the invariant measure $\gamma_\infty$, with an operator quasinorm
 that depends only on the dimension and the matrices $Q$ and $B$.
\end{theorem}
In other words, the inequality
\begin{equation} \label{thesis-mixed-Di}  
\gamma_\infty
\{x\in\R^n : \mathcal H_*f(x) > \alpha\} \le \frac{C}\alpha\,\|f\|_{L^1( \gamma_\infty)},\qquad \text{ $\alpha>0$,}
\end{equation}
holds for all functions $f\in L^1 (\gamma_\infty)$,
with $C=C(n,Q,B)$.

For large values of the time parameter, 
we also obtain a refinement of this result.
Indeed, we prove
in Proposition \ref{propo-mixed-glob-enhanced}
 that \begin{equation}\label{Eldan-est}
\gamma_\infty
 \left\{x\in\R^n : \sup_{t> 1}|\mathcal H_t f(x)| > \alpha\right\} \le 
 \frac{C}{ \alpha\
{\sqrt{\log\alpha}}}\,
\end{equation}
for large $\alpha>0$ and all normalized functions 
$f\in L^1 (\gamma_\infty)$. Here  $C=C(n,Q,B)$, 
and this estimate
is shown to be sharp. It cannot be extended 
to $\mathcal H_*$, since the 
maximal operator corresponding to small values 
of $t$ 
only satisfies the ordinary weak type 
inequality.
This sharpening  is not surprising, in the light of  some recent results for the standard case  $Q=I$ and $B=-I$ by   Lehec \cite{Lehec}. He proved the following conjecture, recently proposed by 
Ball,  Barthe, Bednorz, Oleszkiewicz and  Wolff
\cite{Barthe}:
For each fixed $t>0$, there exists a function $\psi_t=\psi_t(\alpha)$, with
$\displaystyle\lim_{\alpha\to +\infty} \psi_t (\alpha)=0$, satisfying
\begin{equation} \label{thesis-mixed-Talagrand}  
\gamma_\infty
\{x\in\R^n : |\mathcal H_t f(x)| > \alpha\} \le \frac{\psi_t
(\alpha)}\alpha\,
\end{equation}
for all large $\alpha>0$  and all 
$f\in L^1(\gamma_\infty)$ such that $ \|f\|_{L^1( \gamma_\infty)}=1$.
Lehec proved this conjecture with $\psi_t (\alpha)={C(t)}/{\sqrt{\log\alpha}}$ 
 independent of  the dimension, and this $\psi_t$ is  sharp.
Our estimates depend strongly on the dimension $n$, but on the other hand we estimate the supremum over large $t$.    
\medskip

The history of $\mathcal H_*$           
is quite long 
 and started 
with the first attempts to prove  $L^p$ estimates.
When
 $\big(
\mathcal H_t
\big)_{t> 0}\,$
is  symmetric, i.e., when each operator $\mathcal H_t $ is self-adjoint on $L^2 (\gamma_\infty)$,
then $\mathcal H_*$ is bounded on $L^p (\gamma_\infty)$
for  $1<p\le \infty$,
 as a consequence of the general Littlewood--Paley--Stein theory 
for symmetric semigroups of contractions on $L^p$ spaces \cite[Ch. III]{Stein}.

It is easy to see that the  maximal operator
is unbounded on $L^1 (\gamma_\infty)$. This led, about fifty years ago, to the study of
the weak type $(1,1)$ of $\mathcal H_*$ with respect to $\gamma_\infty$.
The first positive result is due to B. Muckenhoupt  \cite{Muckenhoupt}, who proved the estimate  \eqref{thesis-mixed-Di}
in the one-dimensional case with $Q=I$ and $B=-I$.
  The analogous question in the higher-dimensional case was an open problem until 1983, when
 the third author \cite{Peter} proved  the weak type $(1,1)$ in any finite dimension. 
Other proofs are due to
Men\'arguez,  P\'erez and  Soria
 \cite{Soria} (see also \cite{SoriaCR, Soria-Perez}) and  to  Garc\`ia-Cuerva,  Mauceri,  Meda, Sj\"ogren and  Torrea
\cite{JLMS}. 
Moreover, a different proof of the weak type $(1,1)$ of $\mathcal H_*$, based on a covering lemma
 halfway between covering  results by Besicovitch and Wiener,
was given by Aimar, Forzani and Scotto
 \cite{Scotto}.
 A nice overview of the literature may be found in \cite[Ch.4]{Urbina-monograph}.
 
 In \cite{CCS} the present authors recently considered
 a
 normal
 Ornstein--Uhlenbeck
 semigroup
  in $\R^n$,
that is, we assumed that 
$\mathcal H_t$ is for each $t> 0$
 a  normal operator on $L^2 (\gamma_\infty)$. Under this extra assumption, we proved that
 the associated maximal operator is  of weak type $(1, 1)$ with respect to the invariant measure  $\gamma_\infty$. 
 This extends  earlier work in the non-symmetric framework by  Mauceri and Noselli  \cite{Mauceri-Noselli},
who proved some ten years ago that, if $Q=I$ and 
$B=\lambda (R-I)$ for some  positive $\lambda$
and a real skew-symmetric matrix $R$ generating a periodic group,
then the maximal operator $ \mathcal H_* $  is of weak type $(1,1)$.
 
 \medskip
In Theorem \ref{weaktype1} 
we go beyond the hypothesis of normality.
The proof has a geometric core and relies on the {\it{ad hoc}}
technique developed by the third author in \cite{Peter}.
It is worth noticing that, while the proof in \cite{CCS} required  an analysis of the special case
when $Q=I$ and $B=(-\lambda_1, \ldots, -\lambda_n)$, with $\lambda_j>0$ for $j=1, \ldots, n$, and then the application of   factorization results,
 we apply here directly, avoiding many intermediate steps, the "forbidden zones" technique introduced in \cite{Peter}.  
 
Since the maximal operator $\mathcal H_* $ is trivially bounded from 
$L^{\infty }$ to 
$L^{\infty }$, 
  we obtain   by interpolation the following corollary.

\begin{corollary}\label{cor:strongLp}
The Ornstein--Uhlenbeck maximal operator
$\mathcal H_*$
 is bounded on $L^p (\gamma_\infty)$ for all 
$p>1$.
\end{corollary}

This result improves   Theorem 4.2 in \cite{Mauceri-Noselli},
where the $L^p$ boundedness of 
 $\mathcal H_*$
is proved  for all $p>1$
in the normal framework, under the additional assumption that the infinitesimal generator of  $\big(
\mathcal H_t
\big)_{t> 0}\,$ is a sectorial operator of angle less than $\pi/2$.

\medskip

In this paper we focus our attention on the Ornstein--Uhlenbeck
semigroup
in $\R^n$.
In view of possible applications to  stochastic analysis and to SPDE's,
it would be very  interesting to investigate the case of
 the infinite-dimensional Ornstein-Uhlenbeck maximal operator as well
 (see \cite{Goldys,  Neerven, Carbonaro-Oliver} for an introduction
to the infinite-dimensional setting).
The Riesz transforms
associated to a general 
Ornstein--Uhlenbeck semigroup in $\R^n$  will be considered in a forthcoming paper.
\medskip

The scheme of the paper is as follows.
In  Section
\ref{s:particular} 
we introduce the Mehler kernel $K_t(x,u)$, that is, the integral kernel of $\mathcal H_t$.
 Some  estimates  for the norm and the determinant of $Q_t$ and  related matrices are provided in Section \ref{s:Some auxiliary results}. 
As a consequence, we obtain 
bounds for the Mehler  kernel.
In Section \ref{s:Geometric aspects of the problem} we consider the relevant geometric features of the problem, and
  introduce
 in Subsection \ref{coord-choice}  a 
system of polar-like coordinates. We also express  
 Lebesgue measure in terms of these coordinates.
Sections  \ref{simplifications},
 \ref{s:mixed. t large}, \ref{s:The local case} and \ref{s:The global case}
are devoted to the  proof of Theorem~\ref{weaktype1}.  First,
Section \ref{simplifications}  introduces 
some preliminary simplifications  of the proof; in particular, 
we restrict the variable $x$
to an ellipsoidal annulus.
In Section 
 \ref{s:mixed. t large} 
 we consider
 the supremum in the definition of the maximal operator taken only 
over $t> 1$ and prove the sharp  estimate \eqref{Eldan-est}.
Section \ref{s:The local case} is devoted to the case of  small $t$
under an additional local condition. Finally, in Section \ref{s:The global case}
we treat the remaining case and
conclude the proof of Theorem \ref{weaktype1},
by proving the
estimate \eqref{thesis-mixed-Di} for small $t$ under a global assumption.

\medskip

In the following, we use the ``variable constant convention'', according to which
the symbols $c>0$ and $C<\infty$
will denote constants 
which are not necessarily 
equal 
at 
different 
occurrences. They all  depend only on the dimension and on $Q$ and $B$.
For any two nonnegative quantities $a$ and $b$ we write $a\lesssim b$  instead of $a \leq C b$ 
and $a \gtrsim b$ instead of $a \geq c b$.
The symbol $a\simeq b$ means that both $a\lesssim b$ and
$a \gtrsim b$  hold.

By $\N$ we mean the set of all nonnegative integers. 
 If $A$ is  an $n\times n$ matrix, we write $\|A\|$
for its operator norm on $\R^n$
with the Euclidean norm $|\cdot|$.

 \bigskip

\section{{ The Mehler Kernel  
}}\label{s:particular}

For $t>0$,  
the difference
\begin{equation}\label{eq:diff_q_t}
Q_\infty -Q_t =\int_t^{\infty}e^{sB}Qe^{sB^*}ds
\end{equation}
is a symmetric and strictly positive definite matrix. So is the matrix
\begin{equation}\label{diff-oper-inv}  
Q_t ^{-1} -  Q_\infty^{-1} = Q_t^{-1}( Q_\infty -Q_t )Q_\infty^{-1},
\end{equation}
and we can define
\begin{equation}\label{def:Dtx}
D_t =
(Q_t^{-1}-Q_\infty^{-1}
)^{-1} Q_t^{-1} 
e^{tB}\,.
\end{equation}

Then formula  \eqref{Kolmo}, the definition of the Gaussian measure  and some elementary computations yield
\begin{align}
\mathcal H_t&
f(x)
=
(2\pi)^{-\frac{n}{2}}
(\text{det} \, Q_t)^{-\frac{1}{2} }
\int  
f(e^{tB}x-y)
\exp\left[
{-\frac12 \langle Q_t^{-1}y,y\rangle}\right]
dy\notag\\
&=
\Big(
\frac{\text{det} \, Q_\infty}{\text{det} \, Q_t}
\Big)^{{1}/{2} }
\exp\left[
{\frac12 \langle Q_t^{-1} e^{tB}x,  D_t x-e^{tB}x\rangle}\right]
\notag\\
&\quad \times \int  
f(u) \,
\exp\left[
{\frac12 
\langle (Q_\infty^{-1}-
Q_t^{-1}) (u-D_t x)\,,\, u-D_t x\rangle}\right]
d\gamma_\infty (u)\,,
\label{def:Htq-nonfinale}
\end{align}
where
we repeatedly used the fact that
$Q_\infty^{-1}-
Q_t^{-1}$
is symmetric.
We  now express the matrix $D_t$ in various ways.

\begin{lemma}\label{lemma:expr_D}
For all $x\in\R^n$  and $t>0$ 
we have
\begin{enumerate}[label=(\roman*)]
\item[{\rm{(i)}}]
$ D_t =
 Q_\infty
 e^{-tB^*} Q_\infty^{-1} 
$;
\item[{\rm{(ii)}}]
$D_t = e^{tB} + Q_t e^{-tB^*}Q_\infty^{-1}$.
 \end{enumerate}

\end{lemma}

\begin{proof}

{\rm{(i)}}
Formulae \eqref{eq:diff_q_t}  and \eqref{defQt} imply
\begin{equation}
\label{Metafune}Q_\infty- Q_t=
e^{tB} Q_\infty e^{tB^*}
\end{equation}
(see also \cite[formula (2.1)]{MPRS}).
From    \eqref{def:Dtx} and \eqref{diff-oper-inv} it follows that
\[ {D_t  =
 Q_\infty
 ( Q_\infty-Q_t)^{-1}\,
e^{tB},}\] and combining this with \eqref{Metafune}  we arrive at {\rm{(i)}}.

{\rm{(ii)}}
Multiplying \eqref{Metafune} by $ e^{-tB^*} Q_\infty^{-1}$ from the right, we 
obtain
 \begin{align*}
 Q_\infty e^{-tB^*}Q_\infty^{-1}- Q_t  e^{-tB^*}Q_\infty^{-1} = e^{tB},
\end{align*}
and {\rm{(ii)}} now follows from {\rm{(i)}}.
\end{proof}

By means of   {\rm{(i)}}  in this lemma, we can define  $D_t$ for all 
$t\in\R $, and they will form a one-parameter group of matrices.

\smallskip

Now {\rm{(ii)}}  in Lemma \ref{lemma:expr_D}
yields  \begin{align*}
 \langle Q_t^{-1} e^{tB}x,  D_t x-e^{tB}x\rangle=
 \langle Q_t^{-1} e^{tB}x,   Q_t e^{-tB^*}Q_\infty^{-1}x \rangle
=\langle   Q_\infty^{-1}x ,x \rangle.
 \end{align*}
Thus \eqref{def:Htq-nonfinale} may be rewritten as
\begin{align*}
 \mathcal H_t
f(x) &=
 \int
K_t
(x,u)\,
f(u)\,
 d\gamma_\infty(u)  
  \,,
\end{align*}
where  $K_t$ denotes  the Mehler  kernel,
given by 
\begin{align}\label{defKR}
&K_t (x,u)\notag\\
&=
\Big(
\frac{\text{det} \, Q_\infty}{\text{det} \, Q_t}
\Big)^{{1}/{2} }
\exp
{
\big( R(x)\big)}
\times\exp \Big[
{-\frac12 
\left\langle (
Q_t^{-1}-Q_\infty^{-1}) (u-D_t x) \,,\, u-D_t x\right\rangle}\Big]
\,\qquad
\end{align}
for $x,u\in\R^n$. Here 
we  introduced
the quadratic form 
\begin{equation*}
R(x) ={\frac12 \left\langle Q_\infty^{-1}x ,x  \right\rangle}, \qquad\text{$x\in\R^n$.}
\end{equation*}
 
\bigskip

\section{Some auxiliary results}\label{s:Some auxiliary results}
In this section we collect some preliminary bounds, which will be essential for the sequel.
\begin{lemma}\label{expsB-bounded}
For $s>0$ and for all $x\in \R^n$
the matrices $ D_{s }$ and $ D_{-s }= D_{s }^{-1}$ satisfy
\begin{equation*}
  e^{cs}|x| \lesssim |D_s\, x|  \lesssim   e^{Cs} |x|,
\end{equation*}
and \begin{equation*}
  e^{-Cs}|x| \lesssim |D_{-s}\, x|  \lesssim   e^{-cs} |x|.
\end{equation*}
This also holds with  $ D_{s }$ replaced by  $e^{-sB}$ and $e^{-sB^*}$.
 \end{lemma}

 \begin{proof}
   We make a Jordan decomposition of  $B^*$, thus writing it as the
sum  of a complex diagonal matrix and a triangular, nilpotent matrix, which
commute with each other.   This leads to  expressions for  $e^{-sB^*}$
and  $e^{sB^*}$,
and since  $B^*$ like  $B$ has only eigenvalues with negative real parts,
we see that 
\begin{equation}\label{est:2-eBs}
\|e^{-sB^*} \| \lesssim e^{Cs} \qquad \text{ and } \qquad
\|e^{sB^*} \| \lesssim e^{-cs}.  
\end{equation}

From {\rm{(i)}} in Lemma \ref{lemma:expr_D}, we now get the claimed upper estimates for $D_{\pm s}$.
To prove the lower estimate for $D_s$, we write
\begin{equation*}
  |x| =  |D_{-s}\, D_s\, x| \lesssim e^{-cs} |D_s\, x|. 
\end{equation*}
The other parts of the lemma are completely analogous.
 \end{proof}

In the following lemma, we collect  estimates 
of some basic quantities related to the matrices  $Q_t$. 
\begin{lemma}\label{stimadet}
For all $t>0$ we have
\begin{enumerate}[label=(\roman*)]
\item[{\rm{(i)}}]
$\det{ \, Q_t}
\simeq
(\min(1,t))^{n}$;
\item[{\rm{(ii)}}]
$\| 
Q_t^{-1}\|\simeq (\min (1,t))^{-1}$;
\item[{\rm{(iii)}}]
 $\|
Q_\infty-Q_t
\|
\lesssim e^{-ct}$;
\item[{\rm{(iv)}}]
$\|
Q_t^{-1}-Q_\infty^{-1}
\|\lesssim {t}^{-1}\,{e^{-ct}}$;
\item[{\rm{(v)}}]
$\|\left(
Q_t^{-1}-Q_\infty^{-1}
\right)^{-1/2}\|\lesssim {t^{1/2}}\, e^{Ct}
$.
\end{enumerate}
 \end{lemma}

\begin{proof} 
{\rm{(i)}}
 and {\rm{(ii)}} Using \eqref{est:2-eBs}, we see  that for each $t>0$
 and for all $v\in\R^n$
\begin{align*}
\langle 
Q_{t} v ,v\rangle
&=
\left\langle \int_0^t 
e^{sB} Q e^{sB^*}v ds,v\right\rangle
=\int_0^t 
\langle 
 Q^{1/2} e^{sB^*}v ,
 Q^{1/2} e^{sB^*} v\rangle ds\notag
 \\
 &=
 \int_0^t 
\big|
 Q^{1/2} e^{sB^*}v 
\big|^2
 ds\simeq
 \int_0^t 
\big|
 e^{sB^*}v 
\big|^2
 ds\notag
\\
&\lesssim   \int_0^t
 e^{-cs} ds \,|v|^2
 \simeq \min(1,t)\, |v|^2.
 \end{align*}
\\
Since $\| \left(e^{s B^*}\right)^{-1}\| = \| e^{-s B^*}\| \lesssim e^{Cs}$,
there is also a lower estimate
\begin{align*}
 \int_0^t 
&\big|
 e^{sB^*}v 
\big|^2
 ds
\gtrsim
 \int_0^t
 e^{-Cs} ds
\, |v|^2
 \simeq \min(1,t) |v|^2.
 \end{align*}
Thus any  eigenvalue of $Q_t$ 
has order of magnitude 
$\min(1,t)$, and  {\rm{(i)}} and {\rm{(ii)}} follow.
\\
{\rm{(iii)}}
From the definition of $Q_t$ and \eqref{est:2-eBs}, we get
$$\|
Q_\infty
-Q_t
\|
=\left\|
\int_t^\infty e^{sB}Qe^{sB^*}ds\right\|\lesssim
e^{-ct}.$$
\\
{\rm{(iv)}}
Using now {\rm{(ii)}} and  {\rm{(iii)}},  we have
\begin{align*}
\|
Q_t^{-1}-Q_\infty^{-1}\|
&=\|Q_t^{-1}(Q_\infty-Q_t)
Q_\infty^{-1}
\|
\lesssim
\|Q_t^{-1}\|\, \|Q_\infty-Q_t\|
\\
&\lesssim
 (\min (1,t))^{-1}\, e^{-ct} \lesssim t^{-1}\, e^{-ct}.
\end{align*}
{\rm{(v)}}
Since $\|A^{1/2}\| = \|A\|^{1/2}$ for any symmetric positive definite matrix 
$A$, we consider $(Q_t^{-1}-Q_\infty^{-1})^{-1}$, which can be rewritten as
\begin{align} \label{difference}
(Q_t^{-1}-Q_\infty^{-1})^{-1}=
(Q_\infty^{-1}(Q_\infty-Q_t)
Q_t^{-1})^{-1}
=Q_t(Q_\infty-Q_t)^{-1}Q_\infty.
\end{align}
It follows from \eqref{Metafune}
that $(Q_\infty- Q_t)^{-1}=
e^{-tB^*} Q_\infty^{-1}e^{-tB},$
so that
\begin{equation*}\|(Q_\infty- Q_t)^{-1}\|\lesssim
e^{Ct},
\end{equation*}
as a consequence of (3.2).
Inserting this and the simple estimate  $\|Q_t \| \lesssim t$
in \eqref{difference}, we obtain 
$\|(Q_t^{-1}-Q_\infty^{-1})^{-1}\| \lesssim t e^{Ct}$, and {\rm{(v)}}
 follows.
\end{proof}

\begin{proposition}\label{stima-prep-loc-t-large}
For $t\ge 1$ and $w\in\R^n$,  we have
\begin{align*}
\langle (Q_t^{-1}- Q_\infty^{-1}
)
D_t 
w,\, D_t 
w
\rangle&	\simeq |w|^2.
  \end{align*}
\end{proposition}
\begin{proof}
By \eqref{def:Dtx} and   Lemma \ref{lemma:expr_D} {\rm{(i)}}
 we have
\begin{align*}
\langle (Q_t^{-1}- Q_\infty^{-1}
)
D_t 
w,\, D_t 
w
\rangle
=&
\langle
Q_t^{-1} 
e^{tB}
w\,,\, 
  Q_\infty
 e^{-tB^*} Q_\infty^{-1} \,w
 \rangle \\
 =&
\langle
 Q_\infty Q_t^{-1} 
e^{tB}
w\,,\, 
 e^{-tB^*} Q_\infty^{-1} \,w
 \rangle.
\end{align*}
 Since
 $Q_\infty Q_t^{-1}=I+(Q_\infty -Q_t)Q_t^{-1}$,
this leads to 
\begin{multline*}
\langle (Q_t^{-1}- Q_\infty^{-1}
)
D_t 
w,\, D_t 
w
\rangle 
\notag \\
 =
\langle
e^{tB}
w\,,\, 
 e^{-tB^*} Q_\infty^{-1} \,w
 \rangle
+
\langle 
( Q_\infty- Q_t)  Q_t^{-1}
e^{tB}
w\,,\, 
 e^{-tB^*} Q_\infty^{-1} \,w
 \rangle
\notag \\
 = \langle Q_\infty^{-1} w,   w\rangle
+
\langle 
e^{-tB} 
( Q_\infty- Q_t)  Q_t^{-1}
e^{tB}
w\,,\, 
 Q_\infty^{-1} \,w
 \rangle.
 \end{multline*}
 Here
 $ \langle Q_\infty^{-1} w,   w\rangle \simeq |w|^2$.
Using  \eqref{eq:diff_q_t} and then the definition  of $ Q_\infty$, 
we observe that the last term  can be written as
\begin{align}
&
\left\langle
 \int_t^\infty 
 e^{(s-t)B}Qe^{(s-t)B^*}
\,ds\:e^{tB^*} \,  Q_t^{-1}
e^{tB}
w \,,\, 
 Q_\infty^{-1} \,w
\right\rangle
  \notag \\
&=
\big\langle
 Q_\infty
 \,e^{tB^*}   \,  Q_t^{-1}
e^{tB}
w\,  \,,\, 
 Q_\infty^{-1} \,w
 \big\rangle
   \notag \\
&=
\langle
 \,e^{tB^*}   \,  Q_t^{-1}
e^{tB}
w\,  \,,\, 
\,w
 \rangle
   \notag \\
   &=
 \,   \,  \big|Q_t^{-1/2}
e^{tB}
w
\big|^2.
\notag
\end{align}
Since 
$\big|Q_t^{-1/2}
e^{tB}
w
\big|^2
\lesssim
|w|^2
$ for $t\ge 1$ by Lemmata 
\ref{expsB-bounded}
and \ref{stimadet}\,{\rm{(ii)}},
 the proposition follows.\end{proof}

We finally give estimates of the kernel $K_t$, for small and large values of
$t$.
When $t\le 1$, one has 
$\|(Q_t^{-1}-Q_\infty^{-1})^{1/2}\| \simeq t^{-1/2} $  and  
$\|(Q_t^{-1}-Q_\infty^{-1})^{-1/2}\| \simeq t^{1/2}$, by (iv) and (v) in 
Lemma \ref{stimadet}.
Combined with \eqref{defKR},  this implies
\begin{equation}\label{litet}
   \frac{ e^{R( x)}}{t^{n/2}}\exp\big(-C\,\frac{|u-D_t \,x |^2}t\big)
 \lesssim   K_t(x,u)
\lesssim  \frac{ e^{R( x)}}{t^{n/2}} \exp\big(-c\,\frac{|u-D_t\, x |^2}t\big),
\quad \; 0 < t\le 1.
\end{equation}

  \begin{lemma}\label{w}
For $t\geq 1$ and
$x,u\in\R^n$,  we have
\begin{align}\label{stort} 
e^{R(x)}
\exp
\Big[
-C
\big| D_{-t}\,u- x
\big|^2
\Big]
\lesssim
K_t (x,u)
&\lesssim 
e^{R(x)}
\exp
\Big[
-
c\big|
D_{-t}\,u- x\big|^2
\Big].         
\end{align}
\end{lemma}

\begin{proof}
This follows from \eqref{defKR}, if we write
$u-D_t x=
D_t (D_{-t}\,u- x)$
and apply Proposition \ref{stima-prep-loc-t-large}
with $w=D_{-t}\,u- x$.
\end{proof}

\bigskip

\section{Geometric aspects of the problem}\label{s:Geometric aspects of the problem}

\subsection{A  system of adapted polar  coordinates.}\label{coord-choice}

We first  need a technical lemma.
\begin{lemma}\label{elenco-formule}
For all $x$
in $\R^n$ and $s\in\R$, we have
\begin{align}
&\langle  B^* Q_\infty^{-1} x, 
x\rangle =
-
\frac12\, |Q^{1/2}\, Q_\infty^{-1} x|^2;
\label{vel-1}
\\
&\frac{\partial}{\partial s}
 D_sx 
 =
- Q_\infty
 B^*\,Q_\infty^{-1} 
D_s x = -Q_\infty e^{-sB^*}B^*Q_\infty^{-1} x;
 \label{vel-3}
\\
&\frac{\partial}{\partial s}
R\big( D_sx \big)
 =\frac12\,
\big|
Q^{1/2}   Q_\infty^{-1}   D_s x\big|^2
\simeq
\big| 
D_s x\big|^2 .
 \label{vel-4}
\end{align}

\end{lemma}

\begin{proof}
To prove \eqref{vel-1}, we use the definition of  $Q_\infty$ to  write for any $z\in \R^n$
\begin{align*}
\langle B^* z, Q_\infty z\rangle
&=
\int_0^\infty
\langle B^* z,
e^{sB}\,Q\,
e^{sB^*}z\rangle \,ds\\
&=
\int_0^\infty
\langle e^{sB^*}\, B^* z,
\,
Q\, e^{sB^*} z \rangle \,ds\\
&=\frac12\,\int_0^\infty\frac{d}{ds}\langle e^{sB^*}\,  z,\,
Q\, e^{sB^*} z \rangle\, ds\\
&=
-\frac12\, |Q^{1/2}\, z|^2.
\end{align*}
Setting $z=Q_\infty^{-1} x $,
we get \eqref{vel-1}.

Further,
\eqref{vel-3} easily  follows 
if we observe that
\begin{align*}
\frac{\partial}{\partial s}
 D_s x 
=
\frac{\partial}{\partial s}
\left(
 Q_\infty
 e^{-sB^*} Q_\infty^{-1} 
x
\right)
=
- Q_\infty
 B^*\,   Q_\infty^{-1} 
  Q_\infty
 e^{-sB^*}
 Q_\infty^{-1} 
 x
=- Q_\infty
 B^*\,Q_\infty^{-1} 
D_s  x.
\end{align*}

Finally, we get by means of \eqref{vel-3} and  \eqref{vel-1}
\begin{align*}
\frac{\partial}{\partial s}
R\left(D_s x \right)
&=
\frac12
\frac{\partial}{\partial s}
\langle Q_\infty^{-1/2} D_s  x, Q_\infty^{-1/2} D_s  x
\rangle
\\
&=-
\langle Q_\infty^{-1/2} Q_\infty B^*
Q_\infty^{-1} D_s x,  Q_\infty^{-1/2}  D_s x
\rangle
\\
&
=
\frac12
\big|
 Q^{1/2} 
 Q_\infty^{-1}D_s  x\big|^2\,,
 \end{align*}
 and
\eqref{vel-4}
is verified.
\end{proof}
We observe here that an integration 
of \eqref{vel-3}
leads to 
 \begin{equation}\label{xminusDtx}
    |x- D_t \,x| \lesssim t\, |x|,
  \qquad
  0\le t\le 1.
  \end{equation}

Fix now $\beta>0$ and consider the ellipsoid
\begin{equation*}
\ellipses_\beta
=\{x\in\R^n:\, R(x)=
\beta\}
\,.\end{equation*}
 As a consequence of \eqref{vel-4}, 
the map
$s\mapsto R(D_s z)$ is strictly increasing for each $0 \ne z\in\R^n$.
Hence any $x\in\R^n,\, x\neq 0$, 
can be written uniquely as 
\begin{equation}\label{def-coord}
x=D_s \tilde x 
\,,
\end{equation}
for some $\tilde x\in \ellipses_\beta$ 
and $s\in\R$. We consider $s$ and $\tilde x$
as the polar coordinates of $x$. Our estimates in what follows will be uniform in $\beta$.

\medskip

Next, we shall write Lebesgue measure in terms of these polar coordinates.
 A normal vector to the surface  $E_\beta$
at the point $\tilde x \in E_\beta $ is
 ${\bf{N}} (\tilde  x)=Q_\infty^{-1}\tilde x
 $, and
 the tangent hyperplane at   $\tilde x$ 
is ${\bf{N}}(\tilde x)^\perp$. 
For   $s > 0$ the  tangent hyperplane of the surface  $D_s\ellipses_\beta = \{D_s \tilde x: \tilde x\in \ellipses_\beta\}$
at the point  $D_s\tilde x$ is  
$D_s({\bf{N}}(\tilde x)^\perp)$, and a normal to  $D_s\ellipses_\beta$
at the same point is  
$w = (D_s^{-1})^*
({\bf{N}}(\tilde x))=D_{-s}^*
Q_\infty^{-1}\tilde  x  = Q_\infty^{-1}e^{sB} \tilde  x$.

The scalar product of $w$ and the tangent of the curve $s \mapsto D_s\tilde x$
at the point   $D_s\tilde x$ is, because of  \eqref{vel-3} and \eqref{vel-1},
\begin{align}\label{transv}
 & \left\langle \frac \partial { \partial s} D_s \tilde x, w \right \rangle \\& =
-\langle Q_\infty e^{-sB^*}B^* Q_\infty^{-1}  \tilde x,\, Q_\infty^{-1}e^{sB} \tilde  x\rangle = - \langle B^* Q_\infty^{-1}  \tilde x,  \tilde  x\rangle = 
\frac12\, |Q^{1/2}\, Q_\infty^{-1}  \tilde  x|^2 > 0. \notag
\end{align}

Thus the curve  $s \mapsto D_s\tilde x$ is  transversal to each surface
  $D_s\ellipses_\beta$. Let $dS_s$ denote the area measure of 
  $D_s\ellipses_\beta$. Then Lebesgue measure is given in terms of our polar
coordinates by
\begin{align}\label{def:leb-meas}
dx=
H &(s,
\tilde x)
\, dS_s(D_s \tilde x)\,ds,
\end{align}
where
 \begin{equation*}    
H(s,\tilde x)= 
\left\langle \frac \partial { \partial s} D_s \tilde x, 
\frac w {|w|}\right \rangle =
 \frac{ |Q^{1/2}\, Q_\infty^{-1} \tilde x |^2
}{2\  | Q_\infty^{-1} e^{sB} \tilde x  |}.
\end{equation*}

To see how $dS_s$ varies with $s$, we take a continuous function
$\varphi =\varphi(\tilde x) $ on  $\ellipses_\beta$ and extend it to 
$\Bbb R^n \setminus \{0 \}$ by writing 
$\varphi(D_s \tilde x) =\varphi(\tilde x) $. For any $t>0$ and small 
$\varepsilon >0$, we define the shell
\begin{align*}
  \Omega_{t,\varepsilon} = \{D_s \tilde x: t<s<t+ \varepsilon, \;\tilde x \in
\ellipses_\beta\}.
\end{align*}
Then $ \Omega_{t,\varepsilon}$ is the image under $D_t$ of 
 $ \Omega_{0,\varepsilon}$, and the Jacobian of this map is
$\det D_t = e^{-t\tr B}$. Thus
\begin{align*}
  \int_{\Omega_{t,\varepsilon}} \varphi(x)\,dx =
 e^{-t\tr B}  \int_{\Omega_{0,\varepsilon}} \varphi(D_t x)\,dx,
\end{align*}
which we can rewrite as
\begin{align*}
&  \int_{t<s<t+ \varepsilon} \int_{\tilde x \in \ellipses_\beta}
\varphi(\tilde x)\, H(s,\tilde x)\, dS_s(D_s \tilde x)\,ds \\ =\, &
 e^{-t\tr B}\, 
 \int_{0<s< \varepsilon} \int_{\tilde x \in \ellipses_\beta}
\varphi(\tilde x)\, H(s,\tilde x)\, dS_s(D_s \tilde x)\,ds.
\end{align*}
Now we divide by $\varepsilon$ and let $\varepsilon \to 0$, getting
\begin{align*}
  \int_{ \ellipses_\beta}
\varphi(\tilde x)\, H(t,\tilde x)\, dS_t(D_t \tilde x) =
e^{-t\tr B}\,  \int_{ \ellipses_\beta}
\varphi(\tilde x)\, H(0,\tilde x)\, dS_0( \tilde x).
\end{align*}
Since this holds for any $\varphi$, it follows that
\begin{align*}
  dS_t(D_t \tilde x) = e^{-t\tr B}\,
\frac{H(0,\tilde x)}{H(t,\tilde x)}\, dS_0( \tilde x). 
\end{align*}
Together with \eqref{def:leb-meas},
this  implies the following result.

\begin{proposition} \label{propo:equiv-meas-area}
The Lebesgue measure in $\Bbb R^n$ is given in terms of polar coordinates 
$(t, \tilde x)$ by
\begin{align}\label{def:leb-meas-pulita}
  dx =
e^{-t\tr B}\, \frac{ |Q^{1/2}\, Q_\infty^{-1} \tilde x |^2}
{2\,| Q_\infty^{-1} \tilde x  |}\,
 dS_0( \tilde x)\,dt\,.
\end{align}
\end{proposition}

\medskip

We also need   estimates of  the distance between two points in terms of the polar  coordinates. The following result is a generalization of Lemma 4.2 in \cite{CCS}, and its proof is analogous.
\begin{lemma}\label{lemma-peter-coord}
Fix $\beta > 0$.
 Let  $x^{(0)},\; x^{(1)}\in \R^n
 \setminus \{ 0\} $ 
and assume $ R(x^{(0)}) > \beta/2$.
Write  $$x^{(0)} = D_{ s^{(0)}}
(\tilde x^{(0)}
)\qquad
 \text{ 
and   }
\qquad x^{(1)} = D_{s^{(1)}}(\tilde x^{(1)}) 
$$ with  $s^{(0)}$, $s^{(1)}\in \R$ and
$\tilde x^{(0)},\; \tilde x^{(1)} \in E_\beta$. 
\\
\begin{enumerate}[label=(\roman*)]
\item[{\rm{(i)}}]
Then
\begin{equation}
  \label{lem1}
  \big|x^{(0)} - x^{(1)}\big| \gtrsim c\,  
  \big|\tilde x^{(0)} - \tilde x^{(1)}\big|
  .\end{equation}
\item[{\rm{(ii)}}]
If also $s^{(1)} \ge 0$, then 
\begin{equation}
  \label{lem2}
  \big|x^{(0)} - x^{(1)}\big| \gtrsim c\,\sqrt\beta\,|s^{(0)} -s^{(1)}|.
\end{equation}
\end{enumerate}
\end{lemma}

\begin{proof} 
Let $\Gamma: [0,1] \to \R^n \setminus \{0\}$ be a
differentiable curve with  $\Gamma(0) =x^{(0)}$ and  $\Gamma(1) =x^{(1)}$.
It suffices to bound  the length of any such curve from
below by the right-hand sides of \eqref{lem1} and \eqref{lem2}.

For each $\tau \in [0,1]$, we write  
$$\Gamma(\tau) = D_{s{(\tau)}}\,\tilde x{(\tau)} 
,
$$ 
with  $\tilde x(\tau) \in E_\beta$ and  $\tilde x{(i)}= \tilde x^{(i)}$,
 $s{(i)}= s^{(i)}$
for $i=0,1$.
Thus
\begin{eqnarray*}
 \Gamma'(\tau) &=
-s'(\tau) \,\frac{\partial}{\partial s}\,
{D_{s}}_{\big|s=s{(\tau)}}
\,\tilde x{(\tau)} 
+
D_{s{(\tau)}}
\tilde x' (\tau).
 \end{eqnarray*}
 The group property of $D_s$ implies that
 $$\frac{\partial}{\partial s}
{D_{s}}_{\big|s=s{(\tau)}}=
{D_{s(\tau)}}
\frac{\partial}{\partial s}
{D_{s}}_{\big|s=0},$$
and so
\begin{eqnarray}\label{def:gammatau}
 \Gamma'(\tau) &= {D_{s(\tau)}}  v,
 \end{eqnarray}
with
\begin{eqnarray*}
v&=
 -s'(\tau) \,\frac{\partial}{\partial s}\,
{D_{s}}_{\big|s=0}
\,\tilde x{(\tau)} 
+
\tilde x' (\tau).
 \end{eqnarray*}
The vector 
$\tilde x' (\tau)$ is tangent to $E_\beta$
and thus orthogonal to ${\bf{N} }(\tilde x)$. Then \eqref{transv}
(with $s=0$)  implies that the angle between
$\frac{\partial}{\partial s}
{D_{s}}_{\big|s=0} \tilde x (\tau)$
and
$\tilde x'(\tau)$ is larger than some positive constant.
It follows that
\begin{align}\label{ineq:v2}
|v|^2 
\gtrsim
|s'(\tau)|^2\,
\Big|
\frac{\partial}{\partial s}\,
{D_{s}}_{\big|s=0}
\,\tilde x{(\tau)} 
\Big|^2
+
\big|
\tilde x' (\tau)\big|^2
\gtrsim
|s'(\tau)|^2\,
\beta+
\big|
\tilde x' (\tau)\big|^2,
 \end{align}     
where we also used the fact that, by \eqref{vel-3}, 
\begin{equation*}
\Big|
\frac{\partial}{\partial s}\,
{D_{s}}_{\big|s=0}
\,\tilde x{(\tau)} 
\Big|\simeq
|\tilde x(\tau)|
\simeq
\sqrt\beta.
\end{equation*}
Since
 \begin{eqnarray*}
|v|
&= \big|{D_{-s(\tau)}}   \Gamma'(\tau)\big|
\le
 \big\|
{D_{-s(\tau)}}
\big\|\, 
\big|
  \Gamma'(\tau)\big| \lesssim e^{-C\min(s(\tau),0)} \big|  \Gamma'(\tau)\big|
 \end{eqnarray*}
because of Lemma \ref{expsB-bounded}, we obtain from \eqref{ineq:v2}
\begin{align}\label{eq:Peter-X}
\big|
  \Gamma'(\tau)\big|
&\gtrsim
 e^{C \min(s(\tau), 0)}
 \ \big(\sqrt
\beta\,
|s'(\tau)|
+
\big|
\tilde x' (\tau)\big|\big).
 \end{align}

Next, we  derive a lower bound for $ s(0)$; 
 assume first that $s{(0)} < 0$.
The assumption
 $ R(x^{(0)}) > \beta/2$
 implies, together with Lemma \ref{expsB-bounded},
\begin{align*}
   \beta/2 \le
R(D_{s(0)}\, \tilde x^{(0)})
&\lesssim
\big|
D_{ s^{(0)}}\,\tilde x^{(0)}
\big|^2\lesssim
e^{c\, s(0)} \big|
\tilde x^{(0)}
\big|^2
\simeq e^{c\, s(0)}\beta.
 \end{align*}
It follows that
\begin{equation*}
   s{(0)} >  -\tilde s,
 \end{equation*} for some $\tilde s$ with $0 <\tilde s<C$, and this obviously holds also without the assumption $s(0)<0$.
  
Assume now that $ s(\tau) > -\tilde s-1$ 
for all $\tau \in [0,1]$.
Then \eqref{eq:Peter-X}
implies
\begin{equation*}
 \big|
 \Gamma'(\tau)\big| 
 \gtrsim \sqrt\beta\, |s'(\tau)|
 \end{equation*}
and 
\begin{equation*}
 \big|\Gamma'(\tau)\big| \gtrsim  |\tilde x'(\tau)|.
 \end{equation*}
Integrating  these estimates
with respect to $\tau$ in $[0,1]$, we immediately see that one can control  the length of $\Gamma$
from below
 by the right-hand sides of  \eqref{lem1} and \eqref{lem2}. 

If instead   $ s(\tau) \le -\tilde s-1$ for some  $\tau \in [0,1]$,
we can proceed as in the proof of Lemma 4.2 in \cite{CCS}. More precisely, since 
 the image $s([0,1])$ contains the interval 
$[-\tilde s-1, \max (s(0), s(1))]$, 
we can find a closed subinterval  $I$ of $ [0,1] $ whose image $s(I)$
is exactly the interval $[-\tilde s-1, \max (s(0), s(1))]$.
Thus we may use \eqref{eq:Peter-X} to control the  length of $\Gamma$
 by
\begin{equation*}
 \int_0^1 \big|
 \Gamma'(\tau)\big|\,
 d\tau \ge \int_I \big|\Gamma'(\tau)\big|\,
 d\tau \gtrsim 
 \sqrt\beta\, \int_I |s'(\tau)|\,d\tau \ge  \sqrt\beta\: 
\big(\max\, (s(0), s(1))\; + \;\tilde s+1\big).
\end{equation*}
Here
\begin{align*}
\sqrt\beta\: 
\big(\max\, (s(0), s(1))\; + \;\tilde s+1\big)&\gtrsim \sqrt\beta
\gtrsim
\text{diam}\, E_\beta \ge  \big|\tilde x^{(0)} - \tilde x^{(1)}\big|,
\end{align*}
and \eqref{lem1} follows.
Under the additional hypothesis $s(1)\ge 0$ of {\rm{(ii)}},
we have
\[\tilde s\ge
\max\, (-s(0), -s(1))=-\min \,(s(0), s(1)).\]
Then
\begin{align*}
\sqrt\beta\: 
\big(\max\, (s(0), s(1))\; + \;\tilde s+1\big)&\gtrsim
\sqrt\beta\: 
\big(\max\, (s(0), s(1))-\min \,(s(0), s(1))\big)\\
&=
\sqrt\beta\: |s(0)-s(1)|,
\end{align*}
and \eqref{lem2} follows.
\end{proof}

\medskip

\subsection{
The Gaussian measure of a tube}\label{restriction}
We fix a large  $\beta > 0$.
Define for $x^{(1)}\in E_\beta$ and $a>0$ the set
\begin{equation*}
  \Omega = \left\{
  x \in  E_\beta: \left|x - x^{(1)}\right| < a \right\}.
\end{equation*}
This is  a spherical cap of the ellipsoid $E_\beta$, 
centered at $x^{(1)}$.
Observe that $|x| 
\simeq\sqrt \beta $ for  $x \in \Omega$, and that the area of
 $\Omega$ is $|\Omega|\simeq \min\,( a^{n-1},\beta^{(n-1)/2}) $. Then
consider the tube
\begin{equation}   \label{zona}
Z =  \{D_s \tilde x
:s\ge 0,
\;\tilde x \in \Omega \}.
\end{equation}

\begin{lemma}\label{lemma-Peter-forbidden}
There exists a constant $C$ such that $\beta > C$ implies that
the Gaussian
measure 
 of the tube $Z$ fulfills
\begin{equation*}\label{stima-Peter-mu}  
\gamma_\infty (Z)\lesssim
\frac{a^{n-1}}{\sqrt{ \beta}}\, e^{-\beta}.
\end{equation*}
\end{lemma}

\begin{proof}
Proposition \ref{propo:equiv-meas-area}
 yields, since $H(0,\tilde x) \simeq |\tilde x|\simeq \sqrt \beta$,
\begin{align*}   
\gamma_\infty  (Z)  
\simeq
\int_0^\infty 
e^{-s \tr B} \
  e^{-R(D_s \tilde x)}\,
\int_{\Omega}
 {H(0,\tilde x)}
\, dS (\tilde x)\,ds
\lesssim 
\sqrt \beta \,a^{n-1} \int_0^\infty e^{-s \tr B} e^{-R(D_s \tilde x)}\,ds.
\end{align*}
By  \eqref{vel-4} we have
 \begin{align*}
R ({D_s \tilde x})-
R(\tilde x) 
\simeq 
\int_0^s
\big|
D_{s'} \tilde x\big|^2 ds'  \gtrsim
s | \tilde x|^2\simeq s\beta,
\end{align*}
which implies
\begin{align*}
\gamma_\infty  (Z)  
\lesssim    \sqrt{\beta} \ a^{n-1}\, e^{-\beta}
\int_0^\infty 
e^{-s \tr B} \,
  e^{-cs\beta}\,
\ ds
.\end{align*}
Assuming $\beta$ large enough,
one has
$c\beta> -2\tr B$,
and then the last integral is finite and no larger than $C/\beta$.
The lemma follows.
\end{proof}

\bigskip

\section{Some simplifications}\label{simplifications}
In this section, we introduce 
some preliminary simplifications and reductions in the proof of 
 \eqref{thesis-mixed-Di}, i.e., of   Theorem \ref{weaktype1}.
 
  \begin{enumerate}
\item We may assume
that 
$f$ is nonnegative and normalized in the sense that 
 \begin{equation*}
 \|f\|_{L^1( \misgaussk)}=1,
 \end{equation*}
since this involves no loss of generality.
\item We may assume
that  
 $\alpha$ 
is large, $\alpha > C$,  since otherwise  \eqref{thesis-mixed-Di}  and \eqref{Eldan-est}  are trivial.
 \item
 In many cases, we may restrict $x$
  in \eqref{thesis-mixed-Di}  and \eqref{Eldan-est}
   to  the ellipsoidal annulus
\begin{equation}\label{crown}
{\mathcal E_\alpha}=\left\{
x \in\R^n:\, \frac12  \log \alpha\le 
R(x)
\le 2  \log \alpha\,
\right\}.
\end{equation}
To begin with, we can always  forget
the unbounded component of the complement of $\mathcal E_\alpha$, since
\begin{align} \label{restrizione-prima-1}
&\gamma_\infty
\{
x\in
\R^n
:
\,R( x)>
2 \log \alpha
\}\\
\lesssim &
\int_{ R(x)>
 2 \log \alpha }
\exp
(-R(x)  )
dx\, 
\notag
\lesssim
(\log \alpha)^{(n-2)/2}\,\exp (
{-
 2 \log \alpha })
\notag
\lesssim  \frac1\alpha.
 \end{align}
\item
When $t>1$, we may forget also 
the inner region
where $R(x)<\frac12  \log \alpha$.
Indeed, from \eqref{stort}
we get, if  $(x,u)\in \R^n\times\R^n$ with  $R(x) < \frac12 \log \alpha$,
\begin{equation*}
K_t
(x,u) \lesssim  \, e^{R(x)} 
<  \sqrt{\alpha}
<  \alpha,
\end{equation*}
since $\alpha$ is large.
In other words, 
 for any  $(x,u)\in \R^n\times\R^n$ 
\begin{equation}\label{claimv}
 R(x) < \frac12 \log \alpha\; \qquad \Rightarrow \;\qquad 
K_t
(x,u) \lesssim \alpha,
\end{equation}
for all  $t> 1$.\\
Replacing $\alpha$ by $C\alpha$ for some $C$, we see from  \eqref{restrizione-prima-1} 
and
\eqref{claimv}
that
we can assume
 $x
\in {\mathcal E_\alpha}$
 in the 
proof of  \eqref{thesis-mixed-Di} and \eqref{Eldan-est},
when the supremum  of the maximal operator is taken only 
over $t> 1$.
\end{enumerate} 
Before introducing the last simplification,  we need to 
define a global region 
 \begin{align*}
G&=\left\{
(x,u)\in\R^n\times\R^n\,:  \,|x-u|> \frac{1}{1+|x|}
\right\}
\end{align*}
and a local region
 \begin{align*}
L&=\left\{
(x,u)\in\R^n\times\R^n\,:  \,|x-u|\le \frac{1}{1+|x|}
\right\}.
\end{align*}
Notice that the definition of $G$ and $L$ does not depend on $Q$ and $B$.

\setcounter{enumi}{4}
\begin{enumerate}
\item[(5)]
  When $t\le 1$ and $(x,u) \in G$, we shall see that
  \eqref{claimv} is still valid, and it is again enough to 
consider $x\in{\mathcal E_\alpha}$.
\end{enumerate}
 To prove this, we need a lemma which will also be useful later.
\begin{lemma}\label{lemma-sost-peter}
  If $(x,u)\in G$ and $0<t\le 1$, then
\begin{equation*}
\frac1{(1+|x|)^2} 
\lesssim t^2 |x|^2  +
 |
 u-D_{t} x|^2.
\end{equation*}
\end{lemma}
\begin{proof}
  From the definition of $G$ and \eqref{xminusDtx} we get
\begin{align*}
\frac1{1+|x|}& 
\le  |x- u|
\le
 |x-D_t x| +|D_t x -u|
 \lesssim
t | x
|+|u-D_{t} x|.
\end{align*}
The lemma follows.
\end{proof}

To verify now \eqref{claimv} in the global region with
 $t\le 1$,  we recall from \eqref{litet} that
  \begin{align*}
K_t
(x,u) 
 \lesssim   \frac{e^{R(x)}}{
 t^{n/2}} 
\exp
 \Big(-c\,\frac{|
 u-D_tx|^2}t
 \Big)
.
\end{align*}
It follows from Lemma 
\ref{lemma-sost-peter} that 
\begin{align} \label{or}
 t^2 \gtrsim \frac1{(1+|x|)^4} &
 \qquad  \text{or}  \qquad 
\frac{|u-D_{t} x|^2}t 
\gtrsim \frac1{(1+|x|)^2t}.
\end{align}
The first inequality here implies that
\begin{equation*}
K_t
(x,u)
 \lesssim    e^{R(x)}\, (1+|
 x|)^n
 \lesssim   e^{2R(x)},
\end{equation*}
and \eqref{claimv} follows.
If the second inequality of \eqref{or} holds, we have
\begin{equation*}
K_t
(x,u) 
\lesssim  
\frac{e^{R(x)}}{t^{n/2}} 
\exp
 \left(-\frac c{(1+|x|)^2
 t}
 \right)
 \lesssim    e^{R(x)}\, 
 (1+|x|)^n,
\end{equation*}
and we get the same estimate. Thus 
\eqref{claimv} is verified. 

\medskip

Finally, 
let
\begin{align*}
\mathcal H_*^{G}
f(x)=\sup_{0 < t\le 1}
\left|
 \int
K_t
(x,u)\,\chi_{{G}}(x,u)\,f(u)
\,
 d\gamma_\infty(u) 
 \,
\right|\,,
\end{align*}
and
\begin{align*}
\mathcal H_*^{L}
f(x)=\sup_{0 < t\le 1}
\left|
 \int
K_t
(x,u)\,\chi_{{L}}(x,u)\,f(u)
\,
 d\gamma_\infty(u) 
 \,
\right|\,.
\end{align*}

\bigskip

\section{The case of large $t$}\label{s:mixed. t large}
In this section, we consider
 the supremum in the definition of the maximal operator taken only 
over $t> 1$, and we
prove \eqref{Eldan-est}.

\begin{proposition}\label{propo-mixed-glob-enhanced}
For all functions $f\in L^1 (\gamma_\infty)$ such that
$\|f\|_{L^1( \gamma_\infty)}=1$,
\begin{equation} \label{thesis-mixed-Di-tlarge-v}  
\gamma_\infty
\left\{x : \sup_{t> 1}|\mathcal H_t f(x)|
 > \alpha\right\} \lesssim \frac{1}{\alpha\sqrt{\log \alpha}}, \quad 
\quad \text{ $\alpha>2$.}
\end{equation}
In particular,
the maximal operator 
$$
\sup_{t> 1}|\mathcal H_t f(x)|
$$
 is of weak type $(1,1)$ with respect to the invariant measure
 $\misgaussk
$.

\end{proposition}

\begin{proof}
We can assume that $f\ge 0$. Looking at the arguments in Section
\ref{simplifications}, items (3) and (4), we see that it is suffices to
 consider
points $x\in 
 {\mathcal E_\alpha}$.
For both  $x$ and $u$ we use
 the coordinates introduced in  \eqref{def-coord} with $\beta=\log\alpha$, that is,
\begin{equation*}
x=D_{s}\tilde x,
\qquad
u=D_{ s'}
\tilde u, 
\end{equation*}
where $\tilde x, \tilde u
\in E_{\log \alpha}$ and  $s,s' \in\R$.

From \eqref{stort} we have
\begin{align}\label{ineq:ktxu}
K_t
(x,u)
&\lesssim   \exp
(R( x))\exp
 \big(
- c\,
\big|
D_{-t}u-x\big|^2   
\big)
\end{align}
 for $t> 1$
 and $x,u\in\R^n$.
Since 
$x
\in{\mathcal E_\alpha}$
and 
$D_{-t}u=
D_{s'-t} \tilde u,
$
we can apply  
Lemma~\ref{lemma-peter-coord}~{\rm{(i)}},
getting
\begin{align*}
\big| D_{-t}u-x
\big|
\gtrsim
\big|
\tilde x-\tilde u\big|,
\end{align*}
so that
\begin{align*}
\int
 K_t
(x,u)
f(u)\,\misgausskd
 (u)
&\lesssim
\exp
 \big(
R(D_s \tilde x) \big)
\int 
\exp
 \big(
- c\, \big|
\tilde x-\tilde u\big|^2
\big)\,f(u)\,\misgausskd
 (u).
\end{align*}
In view of \eqref{vel-4},  the right-hand side here is strictly
 increasing in $s$,
and therefore the inequality
\begin{equation}\label{equality-s-alfa}
\exp
 \big(
R(D_s \tilde x) \big)
\int 
\exp
 \big(
- c\, \big|
\tilde x-\tilde u\big|^2
\big)\,f(u)\,\misgausskd
 (u)
>
\alpha
\end{equation} holds
if and only if $s>
s_\alpha (\tilde x)$
for some  function
$\tilde x\mapsto s_\alpha (\tilde x)$,
with equality for $s=s_\alpha (\tilde x)$.
Since $\alpha>2$ and  $\| f\|_{L^1 (\gamma_\infty)}=1$, it follows that $s_\alpha (\tilde x)>0$.

For some $C$,  the set
of points  $x\in {\mathcal E_\alpha}$ where the supremum in  \eqref{thesis-mixed-Di-tlarge-v}
is larger than $C\alpha$ is contained in the set $\mathcal A(\alpha)$
of points $D_s \tilde x\in
 {\mathcal E_\alpha}$ fulfilling \eqref{equality-s-alfa}.
We use Proposition~\ref{propo:equiv-meas-area} to estimate the 
$\gamma_\infty$ measure of this set. Observe that 
$H(0,\tilde x) \simeq |\tilde x|\simeq \sqrt{\log\alpha}$ and that 
 $D_s \tilde x\in {\mathcal E_\alpha}$ implies  $s\lesssim 1$, so that also
$e^{-s\tr B} \lesssim 1$. We get
\begin{align*}
\gamma_\infty (\mathcal A(\alpha)\cap {\mathcal E_\alpha})
&=
\int_{\mathcal A(\alpha)\cap {\mathcal E_\alpha}}
e^{-R(x)}
dx
\\
&\lesssim
 { \sqrt{\log \alpha}}
\int_{E_{\log\alpha}}
\int_{s_\alpha (\tilde x)}^{C}
e^{-R(D_s \tilde x)}
\,
dS(\tilde x)\,ds
\\
&
\lesssim
 { \sqrt{\log \alpha}}
\int_{E_{\log\alpha}}
\int_{s_\alpha (\tilde x)}^{+\infty}
\exp
 \left(
-
{R( D_{s_\alpha (\tilde x)} 
\tilde x)
-c\log\alpha\,
(s-s_\alpha (\tilde x)})
 \right)
 \,ds\,
dS(\tilde x) ,
\end{align*}
where the last inequality follows from 
\eqref{vel-4}, 
since
$
|D_s \tilde x|^2\gtrsim
|\tilde x|^2\simeq \log\alpha.
$
Integrating in $s$, we obtain
\begin{align*}
\gamma_\infty (\mathcal A
(\alpha)\cap{\mathcal E_\alpha})
&\lesssim \frac{1}{\sqrt{\log \alpha}}\int_{E_{\log\alpha}}
\exp
 \big(-R(
D_{{s_\alpha  (\tilde x)}}  
\tilde x)
 \big)
\,dS(\tilde x) .
\end{align*}
Now combine this estimate
with the case of equality in  \eqref{equality-s-alfa} and change  the order of integration, to get
\begin{align*}
\gamma_\infty
 (\mathcal A
 (\alpha)\cap{\mathcal E_\alpha})
&\lesssim
\frac{1}{\alpha \sqrt{\log \alpha}}
\int
 \int_{E_{\log\alpha}}
\exp
 \big(
- c\, \big|
\tilde x-\tilde u\big|^2
\big) \,dS(\tilde x)\,f(u)\,\misgausskd
 (u)\\
&
\lesssim
\frac{1}{\alpha \sqrt{\log \alpha}}
\int
f(u)\,\misgausskd
 (u)\,,
\end{align*}
which proves
 Proposition \ref{propo-mixed-glob-enhanced}.
\end{proof}

Finally,  we show that
 the factor $1/\sqrt{\log\alpha}$ in 
\eqref{thesis-mixed-Di-tlarge-v}  
is sharp.
\begin{proposition}
For any  $t> 1$ and any large $\alpha$,
 there exists 
a function $ f$, normalized in $L^1 (\gamma_\infty)$
 and such that
\begin{equation*} 
\gamma_\infty
\left\{x :|\mathcal H_t f(x)|
 > \alpha\right\} \simeq \frac{1}{\alpha\sqrt{\log \alpha}}\,.
\end{equation*}
\end{proposition}
\begin{proof}
  Take a point
 $z$ with $R(z)=\log\,\alpha$, and let $f$ be (an approximation of) a Dirac measure at the point
 $u=D_t z$.
 Then, as a consequence of \eqref{stort},
$K_t
(x,u)
\simeq   \exp
(R( x))$
in the ball $B(D_{-t}u,1)=B(z,1)$. 
We then have
$\mathcal H_t f(x)=K_t (x,u)\gtrsim \alpha$
in the set
$\mathcal B=\{x\in B(z,1):
R(x)>R(z)
\}$, whose measure is
$$\gamma_\infty\, (\mathcal B)
\simeq e^{-R(z)}\,
\frac{1}{\sqrt{R(z)}}
=
\frac{1}{\alpha\,\sqrt{\log \,\alpha}}.
$$
\end{proof}

\section{The local case  for small $t$}\label{s:The local case}


\begin{proposition}\label{stima1loc}
If 
$(x,u) \in L
$ 
and $0<t\le 1$, then
\begin{equation*}
\big|K_{t}
(x,u)\big|\lesssim
\,\frac{
\exp \big(
R(x) \big)
}{  t^{n/2}
}
\exp{
 \Big(-c\,
 \frac{|
 u-x|^2}{ t} \,
\Big)
}
\,.
\end{equation*}

\end{proposition}

\begin{proof} 
In view of \eqref{litet}, it is enough to show that
\begin{equation}\label{final-aim-stima-loc}
\frac{|u-D_t x|^2}t \ge \frac{|u-x|^2}t -C.
\end{equation}
 
We write
\begin{align*}
  |u-D_t\, x|^2 = |u-x +x-D_t \,x|^2 =  |u-x|^2 +2 \langle u-x, x-D_t\, x\rangle
+ |x-D_t\, x|^2\\
\ge |u-x|^2 -2 |u-x|\,|x-D_t\, x|.
\end{align*}
By \eqref{xminusDtx},
\[ 
 |u-x|\, |x-D_t\, x| \lesssim
 |u-x|\,t\, |x|\le t
\]
since $(x,u) \in L$,   and \eqref{final-aim-stima-loc} follows.
\end{proof}

\begin{proposition}\label{propo-locale}
The maximal operator $\mathcal H_*^{L}
$ is of weak type $(1,1)$ with respect to the invariant measure
 $\misgaussk$.
\end{proposition}
\begin{proof} 
The proof is standard, since 
Proposition \ref{stima1loc} implies
\begin{align*}
\mathcal H_*^{L}
f(x)
&\lesssim
\sup_{0<t\le  1}
\frac{\exp
\big( R(x)\big)}{
t^{n/2}}
 \int
\exp
 \Big(-c\,
 \frac{|
 x-u|^2
 }{ t} \,
\Big)
\,\chi_{L}(x,u)\,f(u)
\,
 d\gamma_\infty(u). 
\end{align*}
The supremum here defines an operator of 
weak type $(1,1)$
 with respect to Lebesgue measure in $\R^n$.
 From this  the proposition follows, cf.
\cite[Section 3]{JLMS}.
\end{proof}

\medskip

\section{The global case for   small $t$}\label{s:The global case}
 In this section, we conclude the proof of 
Theorem \ref{weaktype1}.
\begin{proposition}\label{stima-tipo-debole-misto}
The maximal operator $\mathcal H_*^{G}$
 is of weak type $(1,1)$ with respect to the invariant measure
 $\misgaussk$.
\end{proposition}

\begin{proof}
We take $f$ and $\alpha$ as in items (1) and (2)
of Section \ref{simplifications}.
Then item (5) tells us that we need only consider
$\mathcal H_*^{G} f(x)$ for $x\in\mathcal E_\alpha$.

For $m\in \N$ and $0<t\le 1$, 
we introduce regions $\mathcal S^{m}_t$.
If $m>0$,  we let
\begin{align*}
\mathcal S^{m}_t
&=\left\{(x,u)\in G:
2^{m-1} \sqrt t< |u-D_t x |\le 2^{m} \sqrt t
\,\right\}.
\end{align*}
If $m=0$,
we replace the condition
$2^{m-1} \sqrt t< |u-D_t x |\le 2^{m} \sqrt t$ by
$|
u-D_t x | \le  \sqrt t$.
Note that for any fixed $t\in (0,1]$ these sets form a partition of 
$G$.

In the set ${\mathcal S^{m}_t}$ we have, because of  \eqref{litet},
\begin{align*}
K_t (x,u) &
\lesssim
\frac{
\exp
 (R(x))
}{  t^{n/2}
}
\exp 
\left({
-c{2^{2 m} }  }\right).
\end{align*}
Then setting
\begin{align}\label{ktm1m2}
{\mathcal K}_t^{{m}} 
(x,u)
&=
\frac{
\exp  (R(x))
}{  t^{n/2}
}
\,\chi_{\mathcal S^{m}_t}(x,u),
\end{align}
one has, for all $(x,u)\in G$ and $0<t<1$,
\begin{equation*}
K_t (x,u)\lesssim
\sum_{m=0}^\infty \exp 
\left({
-c{2^{2 m} }  }\right)
{\mathcal K}_t^{{m}} 
(x,u)
\,.
\end{equation*}
Hence, it suffices to  prove
that for $m = 0,1,\dots$ 
\begin{equation}\label{obiettivo-finale}
\gamma_\infty
\left\{
x\in{\mathcal E_\alpha}:
\sup_{0<t\le 1}
 \int
{\mathcal K}_t^{{m}} 
\!
(x,u)\, f(u)\,
\misgausskd
(u) 
\!
>\alpha \right\}
 \lesssim
\frac{2^{Cm}}{\alpha}
,
\end{equation}
for large $\alpha$ and some $C$,
since this will allow summing in $m$
 in the space  $L^{1,\infty}(\gamma_\infty)$.

Fix $m\in \Bbb N$ and assume that
 $(x,u) \in S_t^m$ for some $t \in (0,1]$, so that
$|u-D_t x|
\leq 2^{m}\sqrt{t}$.
Then Lemma~\ref{lemma-sost-peter}
leads to
\begin{align*}
1& \lesssim 
 (1+|x|)^4 
 t^2+(1+|x|)^2 \,
 2^{2m}\,t 
\le
 ((1+|x
|)^2 \,
2^{2m}\,t)^2+
 (1+|x|)^2 \,2^{2m}\,t.
\end{align*}
Consequently, a point $x\in\mathcal E_\alpha$ satisfies
\begin{equation}\label{stima-t-}
 (1+|x|)^2 \,
 2^{2m}\,t \gtrsim 1
\end{equation}
as soon as there exists a point $u$ with $\mathcal K_t^{m}(x,u)\neq 0$,
and then
$t\ge \varepsilon>0$ for some
$\varepsilon=\varepsilon(\alpha,m)
>0$.
Hence the supremum in  \eqref{obiettivo-finale}
will be the same if taken only  over 
$\varepsilon\le t\le 1$,
and it follows that   this supremum is a continuous function
of $x\in  {\mathcal E_\alpha}$.
  
\medskip

To prove \eqref{obiettivo-finale},
the idea, which goes back to \cite{Peter}, is
 to construct a finite sequence of pairwise disjoint balls
$\big(\mathcal B^{(\ell)}\big)_{\ell=1}^{\ell_0}$ in $\R^n$ and a 
finite sequence of sets $\big(\mathcal Z^{(\ell)}\big)_{\ell=1}^{\ell_0}$ in $\R^n$, called forbidden zones. These zones will together cover
 the level set in \eqref{obiettivo-finale}.
We claim that
  \begin{equation}\label{final-subset}
\left\{x\in {\mathcal E_\alpha}:
\sup_{\varepsilon\le t\le 1}
\int {\mathcal K}_t^{{m}}   (x,u)\,f(u)\,\misgausskd (u)\,
\ge \alpha
\right\}\subset \bigcup_{\ell=1}^{\ell_0}\mathcal Z^{(\ell)},
\end{equation}
that for each $\ell$
\begin{align}\label{stima-con-exp} 
&\misgaussk (\mathcal Z^{(\ell)})
\lesssim
\frac{2^{Cm}}{\alpha} 
 \int_{\mathcal B^{(\ell)}}f(u) 
\,\misgausskd
 (u),
\end{align}
and that the $ \mathcal B^{(\ell)}$ are pairwise disjoint. This would imply 
\begin{align*} 
\misgaussk \Big(\bigcup_{\ell=1}^{\ell_0} \mathcal Z^{(\ell)}  \Big)
 \lesssim
\frac{2^{Cm}}{\alpha}\,
 \sum_{\ell=1}^{\ell_0}
  \int_{\mathcal B^{(\ell)}}f(u)\,\misgausskd
 (u)\notag
 \lesssim
\frac{2^{Cm}}{\alpha}
,\end{align*}
and thus also 
\eqref{obiettivo-finale}
and 
Proposition \ref{stima-tipo-debole-misto}.

\medskip

The  sets $\mathcal B^{(\ell)}$ and
 $\mathcal Z^{(\ell)}$
will be introduced by means of a sequence of
points $x^{(\ell)}$,\; $\ell=1,\ldots, \ell_0$,
which we define by recursion.
To start, we choose as $x^{(1)}$
a point 
where the quadratic form $R(x)$ takes its minimal value in the compact set
  \begin{equation*}                
  \mathcal A_1 (\alpha)=
\left\{x\in {\mathcal E_\alpha} :
\sup_{\varepsilon\le t\le 1}
\int 
{\mathcal K}_t^{{m}}   (x,u)\,f(u)\,\misgausskd
\ge \alpha
\right\} .
\end{equation*}
However, should this set be empty, 
   \eqref{obiettivo-finale} 
is immediate.

We now describe the recursion to construct   $x^{(\ell)}$ for $\ell \ge 2$.
Like  $x^{(1)}$, these points will satisfy
\[
\sup_{\varepsilon\le t\le 1}
\int 
{\mathcal K}_t^{{m}}   (x^{(\ell)},u)\,f(u)\,\misgausskd
\ge \alpha.
 \]
Once an  $x^{(\ell)},\;\ell\ge 1 $, is defined, we can thus by continuity 
choose 
 $t_\ell \in [\varepsilon, 1] $ such that
\begin{equation}\label{pha}
 \int {\mathcal K}_{t_\ell}^{{m}}   (x^{(\ell)},u)\,
f(u)\,\misgausskd \ge \alpha.
\end{equation}

Using this  $t_\ell$, we  associate with   $x^{(\ell)}$ the tube
\begin{equation*}
\mathcal Z^{(\ell)} = 
\left\{
D_{ s}\eta
\in \R^n
:\,s\ge 0,\; 
R{( \eta
)}= R(x^{(\ell)}),
\;
| \eta- x^{(\ell)}
|< 
A\, 2^{3m}\, \sqrt{t_{\ell}}\right\},
\end{equation*}
Here the constant $A>0$ is to be determined, depending only on 
$n$,
$Q$ and $B$.

All the $x^{(\ell)}$ will be minimizing points of $R(x)$. To avoid
having them too close to one another, we will not allow  $x^{(\ell)}$  
to be in any $\mathcal Z^{(\ell')}$ with $\ell' < \ell$. More precisely, assuming  
$x^{(1)}, \dots, x^{(\ell)}$ already defined, we will choose   $x^{(\ell+1)}$
as a minimizing point of $R(x)$ in the set
 \begin{equation}
\label{def:set}
  \mathcal A_{\ell+1} (\alpha)=
\left\{x\in {\mathcal E_\alpha}
  \setminus   \bigcup_{\ell'=1}^{\ell} \mathcal Z^{(\ell')}:\,
\sup_{\varepsilon\le t\le 1}
\int {\mathcal K}_t^{{m}}   (x,u)\,f(u)\,d\gamma_\infty(u)
\ge \alpha
\right\},
\end{equation}
provided this set is nonempty.  But if  $\mathcal A_{\ell+1} (\alpha)$
is empty,
the process stops with $\ell_0=\ell$ and \eqref{final-subset} follows.
We will  see that this actually occurs for some  finite $\ell$.

Now assume that  $\mathcal A_{\ell+1} (\alpha) \ne \emptyset$. In order 
to assure that
a minimizing point exists, we must verify that   $\mathcal A_{\ell+1} (\alpha)$
is closed and thus compact, although the $\mathcal  Z^{(\ell')}$ are not open.
To do so, observe that
for $1\le  \ell' \le \ell$,
 the minimizing property of $x^{(\ell')}$ means that there is no point in
$\mathcal A_{\ell'} (\alpha) $ with $R(x) < R(x^{(\ell')})$. Thus we have the inclusions
\begin{equation*}
  \mathcal A_{\ell+1} (\alpha) \subset  \mathcal A_{\ell'} (\alpha)
\subset \left\{x: R(x)\ge R(x^{(\ell')})\right\}, \qquad 1\le  \ell' \le \ell.
\end{equation*}
  It follows that
\begin{multline*}
   \mathcal A_{\ell+1} (\alpha) = \mathcal A_{\ell+1} (\alpha)\,\cap\, \bigcap_{1\le  \ell' \le\ell}
\{x: R(x)\ge R(x^{(\ell')})\} =\\ 
\bigcap_{\ell'=1}^{\ell} \left\{x\in {\mathcal E_\alpha} \setminus   \mathcal Z^{(\ell')}:
R(x)\ge R(x^{(\ell')}), \;    
\sup_{\varepsilon\le t\le 1}
\int {\mathcal K}_t^{{m}}   (x,u)\,f(u)\,d\gamma_\infty(u)
\ge \alpha\right\}.
\end{multline*}
The sets $ \{x\in {\mathcal E_\alpha} \setminus   \mathcal Z^{(\ell')}:
R(x)\ge R(x^{(\ell')})\}$ are closed in view of the choice of $\mathcal Z^{(\ell')}$.
This makes $  \mathcal A_{\ell+1} (\alpha)$  compact, and a minimizing point  
$x^{(\ell+1)}$ can be chosen. Thus the recursion is well defined.

We observe that    \eqref{stima-t-}
applies to $t_\ell$ and
$x^{(\ell)}$, and $|x^{(\ell)}|$ is large, 
so 
\begin{equation}\label{stima-t-prop}
 |x^{(\ell)}|^2\,
  2^{2m}\, t_\ell  \gtrsim 1.
\end{equation}
Further, we define balls
\begin{align*}
\mathcal B^{(\ell)}=&
\{
u
\in \R^n
:\,
|
 u  -D_{t_\ell} x^{(\ell)}
 |
 \le 2^{m} \sqrt{t_\ell}\,
\}\,.
\end{align*}
Because of \eqref{ktm1m2} 
and the definitions of $\mathcal K_t^m$ and ${\mathcal S^{m}_t}$, the inequality
\eqref{pha} implies
\begin{align}\label{mixed-bound-ell-1}
\alpha&\le
\frac
{\exp \left({
R(x^{(\ell)})   }\right)
}{  t_\ell^{n/2}}
\int_{\mathcal B^{(\ell)}
}f(u)\,\misgausskd
 (u).
\end{align}
It remains to verify the claimed properties of
$\mathcal B^{(\ell)}$ and 
$\mathcal Z^{(\ell)}$.
The proof
follows the lines of the proof of Lemma 6.2 in \cite{CCS}, with only slight modifications.
\begin{lemma}\label{lemma:disjoint}
The  balls 
$\mathcal B^{(\ell)}$ 
are pairwise disjoint.
\end{lemma}
\begin{proof}
Two balls
$\mathcal B^{(\ell)}$ 
and $\mathcal B^{(\ell')}$  with $\ell<\ell'$ will be disjoint if
 \begin{equation}\label{stima-distanza-tang}
\big|
D_{t_{\ell}} 
 x^{(\ell)}-
D_{t_{\ell'}} 
 x^{(\ell')}\big|
>
 2^m (\sqrt {t_\ell}+
 \sqrt{ t_{\ell'}}).
\end{equation}

By means of  our polar coordinates
with $\beta=R(x^{(\ell)})$,
we write
\begin{equation*}
x^{(\ell')}
=
D_{s} \tilde x^{(\ell')} 
\end{equation*}
for some $\tilde x^{(\ell')}$ with
$R(\tilde x^{(\ell')})=
R(x^{(\ell)})$ and some $s \in\R$.
Note that  $s\ge 0$, because
 $R(x^{(\ell')})\ge R( x^{(\ell)})$.
Since $x^{(\ell')}$ does not belong to the forbidden zone $ \mathcal Z^{(\ell)}$,
we must have
\begin{equation}\label{hypo1}
|
\tilde x^{(\ell')}-
x^{(\ell)}
|\ge A
 2^{3m} \sqrt {t_\ell}.
 \end{equation}

 We first assume that ${t_{\ell'}}
\ge M\, 2^{4m} \, t_\ell$, for some $M\ge 2$ to be chosen.
 Lemma~\ref{lemma-peter-coord}~{\rm{(ii)}}
 implies
\begin{align*}
\big|
D_{t_{\ell}} 
 x^{(\ell)}-
D_{t_{\ell'}} 
 x^{(\ell')}\big|
 =
 \big|
D_{t_{\ell}} 
 x^{(\ell)}-
D_{t_{\ell'}+s} 
\tilde x^{(\ell')}\big|
\gtrsim  
|x^{(\ell)}|\,( t_{\ell'}+s-t_\ell)
\gtrsim  
|x^{(\ell)}|\,
{ t_{\ell'}}.
\end{align*}
Using our assumption and then \eqref{stima-t-prop}, we get
\begin{equation*}
  |x^{(\ell)}|\,{ t_{\ell'}}\gtrsim 
|x^{(\ell)}|\, \sqrt{M} \, 2^{2m} \sqrt{ t_{\ell}}\, \sqrt{ t_{\ell'}}
\gtrsim \sqrt{M} \, 2^{m}\sqrt{ t_{\ell'}}
\simeq  \sqrt{M}\, 2^{m}\,
( \sqrt{ t_{\ell'}} +\sqrt{ t_{\ell}}). 
\end{equation*}

Fixing $M$ suitably large,
we obtain  \eqref{stima-distanza-tang} from the last two formulae.

It remains to consider the case when  ${t_{\ell'}}
< M\, 2^{4m}\,  t_\ell$. Then
\begin{equation*}
\sqrt{t_{\ell}}
>
\frac{2^{-2m-1}}{\sqrt{M}}
( \sqrt {t_{\ell'}}+\sqrt {t_\ell}).
\end{equation*}
Applying this to  \eqref{hypo1},
 we obtain
\eqref{stima-distanza-tang}
by choosing $A$ so that $A/\sqrt{M}$ is large enough.
\end{proof}

We next verify that the sequence $(x^{(\ell)})$ 
is finite.
For $\ell<\ell'$,   
 we have \eqref{hypo1}, and Lemma \ref{lemma-peter-coord} {\rm{(i)}}
implies
\begin{align*}
\big|
x^{(\ell')}-
 x^{(\ell)}
\big|&
\gtrsim A\,
 2^{3m} \sqrt {t_\ell}.
\end{align*}
Since 
$t_\ell\ge \varepsilon$, we see that the distance 
$\left|
x^{(\ell')}-
x^{(\ell)}
\right|$ 
is bounded below by a positive constant.
But all the $
x^{(\ell)}$
are contained in the bounded set 
$ {\mathcal E_\alpha}$, so they are finite in number.
Thus the set considered in \eqref{def:set}
must be empty for some $\ell$, and the recursion stops.  This implies 
\eqref{final-subset}.

We finally prove 
\eqref{stima-con-exp} .
Observe that the  forbidden zone
$\mathcal Z^{(\ell)}$
is a tube  as  defined in \eqref{zona}, with
 $a=A\, 2^{3m} 
 \sqrt{t_\ell}$ and 
 $\beta=R(x^{(\ell)})$. 
This value of  $\beta$ is large since $x^{(\ell)} \in {\mathcal E_\alpha}$,
and thus we can apply
Lemma 
\ref{lemma-Peter-forbidden}
to  obtain
\begin{align*}
\misgaussk (  \mathcal Z^{(\ell)})\lesssim
\frac{\big(  A 2^{3m} \sqrt{t_\ell}\big)^{n-1}}{
\sqrt{ R(x^{(\ell)})}}\, \exp\left(
{ -{
R(x^{(\ell)})   }}\right)
\notag
. \end{align*}
We bound the exponential here by  means of \eqref{mixed-bound-ell-1}
and observe that $R(x^{(\ell)}) \sim |x^{(\ell)}|^2$, getting
\begin{align*}\misgaussk (\mathcal Z^{(\ell)})\!\!
&\lesssim \frac{1}{\alpha |x^{(\ell)}|
{\sqrt{t_\ell}}} \,
 (A 2^{3m})^{n-1} \,
\int_{\mathcal B^{(\ell)}
}\!f(u)\misgausskd
 (u).
\end{align*}
As a consequence of
\eqref{stima-t-prop},
we obtain
\begin{align*} 
\misgaussk
 (\mathcal Z^{(\ell)})
\lesssim
\frac{2^{m}}{\alpha} \,
 \big(A 2^{3m}\big)^{n-1} \,
\int_{\mathcal B^{(\ell)}}f(u)\,\misgausskd
 (u)\,
\lesssim
\frac{2^{Cm}}{\alpha}\, 
 \, \int_{\mathcal B^{(\ell)}}f(u)\,\misgausskd
 (u),
\end{align*}
 proving 
\eqref{stima-con-exp}.
This  concludes  the proof of Proposition  \ref{stima-tipo-debole-misto}.
\end{proof}

\end{document}